\documentclass[11pt,reqno]{amsart}
 \usepackage{amssymb}
 \usepackage{amsmath}
 \usepackage{mathrsfs}
   \usepackage{amssymb, amsmath, amsfonts}
    \usepackage[title]{appendix}
    \usepackage{mathabx}

    \usepackage{color}

    \usepackage[hidelinks]{hyperref}
    \allowdisplaybreaks
    \usepackage[numbers,sort&compress]{natbib}

    \makeatletter
    \def\tank#1{\protected@xdef\@thanks{\@thanks
     \protect\footnotetext[0]{#1}}}
    \def\bigfoot{

     \@footnotetext}
    \makeatother

    \newcommand{\ea}{\end{array}}

\allowdisplaybreaks
\numberwithin{equation}{section}

\allowdisplaybreaks

\newtheorem{theorem}{Theorem}[section]
\newtheorem{lemma}{Lemma}[section]
\newtheorem{proposition}[theorem]{Proposition}

\def\beq{\begin{equation}}
\def\nneq{\end{equation}}

\def\bthm{\begin{theorem}}
\def\nthm{\end{theorem}}

\def\blem{\begin{lemma}}
\def\nlem{\end{lemma}}
\def\bprf{\begin{proof}}
\def\nprf{\end{proof}}
\def\bprop{\begin{prop}}
\def\nprop{\end{prop}}
\def\brmk{\begin{rem}}
\def\nrmk{\end{rem}}

\def\bexa{\begin{exa}}
\def\nexa{\end{exa}}
\def\bcor{\begin{cor}}
\def\ncor{\end{cor}}

\def\RR{\mathbb{R}}

\def\EE{\mathbb{E}}

\def\cF{\mathcal{F}}

\def\cD{\mathcal{D}}

\newcommand\bp{\mathbb{P}}

\def\ee{{\mathbb E}}
\def\e{{\varepsilon}}

\newcommand{\lc}{\left(}
\newcommand{\rc}{\right)}

\newcommand{\lt}{\left }
\newcommand{\rt}{\right}

\title[Chung's LIL for the linear SFHE at origin]{Chung's LIL for the linear stochastic fractional  heat equation   at origin}

    \author[C. Liu]{Chang Liu}
    \address[]{Chang Liu, School of Mathematics and Statistics,  Wuhan University,  Wuhan, 430072,
    China.}
    \email{changliu0504@163.com}

    \author[R. Wang]{Ran Wang}
    \address[]{Ran Wang, School of Mathematics and Statistics,  Wuhan University,  Wuhan, 430072,
    China.}
    \email{rwang@whu.edu.cn}

    \date{}
    
\begin{document}
    \maketitle

     \noindent {\bf Abstract:}
  Consider  the   linear stochastic fractional heat equation with vanishing initial condition:
$$
  \frac{\partial u (t,x)}{\partial t}=-(-\Delta)^{\frac{\alpha}2}u (t,x)  + \dot{W}(t,x),\quad   t> 0,\,
   x\in\RR,
$$
 where  $-(-\Delta)^{\frac{\alpha}{2}}$ denotes the fractional Laplacian with power $\alpha\in (1,2]$, and the driving noise $\dot W$ is
    a centered  Gaussian field which is white in time and has the covariance of a fractional Brownian motion  with Hurst parameter $H\in\left(\frac {2-\alpha}2,1\right)$.    We  establish Chung's law  of the iterated logarithm for the solution  at $t=0$.

 \vskip0.3cm


\section{Introduction  }
Consider the following linear stochastic fractional  heat  equation (SFHE, for short):
\begin{equation}\label{SFHE}
\begin{cases}
  \frac{\partial  u(t,x)}{\partial t}=-(-\Delta)^{\frac{\alpha}2}u(t,x)+ \dot{W}(t,x),\quad   t>0,\,
   x\in\RR, \\
  u(0,x)\equiv 0.
\end{cases}
\end{equation}
Here,    $-(-\Delta)^{\frac{\alpha}2}$ denotes the fractional Laplacian   with $\alpha\in(1,2]$,
 $W(t,x)$ is a centered Gaussian field with the covariance given by
\begin{equation}
  \ee[W(t,x)W(s,y)]=\frac 12 \left(s\wedge t\right)\left( |x|^{2H}+|y|^{2H}-|x-y|^{2H} \right),
\end{equation}
with
\begin{align}\label{eq const H}
 \frac {2-\alpha}2<  H< 1.
\end{align}
That is, $W$ is a random fields, which is a standard Brownian motion in time and a fractional Brownian motion
(fBm, for short) with Hurst parameter $H$ in space.

According to \cite[Theorem 5.3, Remark 5.4]{B12},
the equation     \eqref{SFHE} has a random field solution if and only if   \eqref{eq const H} holds; See Section 2 for more precise information.  In the latter case,  the solution $\{u(t, x)\}_{t\ge0, x\in \mathbb R}$ can be written as
     \begin{equation}\label{u}
u(t, x)=  \int_0^t \int_{\mathbb{R}}G(t-s, x-y)  W(ds,dy), \ \ \text{a.s.}
\end{equation}
where     $G(t, x)$ is  the Green kernel associated with the operator $-(-\triangle)^{\frac{\alpha}{2}}$ is given by
    \begin{equation}\label{F-Green kernel}
    (\cF G(t,\cdot))(\xi)=e^{-t|\xi|^{\alpha}}, \ \ \xi\in{\RR},\ \ t> 0.
    \end{equation}

Using the idea of the pinned string process from Mueller and Tribe \cite{MT02}, it  can be shown that the solution to the linear stochastic (fractional) heat equation can be decomposed into the sum of two Gaussian random fields: one possessing stationary increments and the other exhibiting smooth sample paths. See  \cite{GSWX2025, HSWX2020, LN2009, KT2019, TX17, WZ2021}, as well as the monographs \cite{DS2026, K, Tudor23}. For example,  when  $\dot W$ is a space-time white noise (i.e., $H=1/2$),    Theorem 3.3 in \cite{K} tells us that for any fixed $x\in \mathbb R$,    the stochastic process $t\mapsto u(t,x)$  can be represented as the sum of a fBm with Hurst parameter $1/4$  and a stochastic  process with  $C^{\infty}(0,\infty)$-continuous trajectories.  Hence,    locally  $t\mapsto u(t,x)$  behaves as a fBm with Hurst parameter $1/4$, and it has the same   regularity  (such as  the H\"older continuity, the  local moduli of continuities,  Chung-type law  of the iterated logarithm) as a fBm with Hurst parameter $1/4$.

For the solution of the  equation \eqref{SFHE},  according to  \cite[Section 3.2]{GSWX2025},  we know that  for any fixed $x\in \mathbb R$, there exists a Gaussian process $\{T(t)\}_{t\ge0}$, which is  independent of $\{u(t,x)\}_{t\ge 0}$ and has a version that is infinitely differentiable on $(0, \infty)$, such that
\begin{align}\label{eq decom}
F:=\kappa^{-1}\left\{u(t,x)+T(t)\right\}_{t\ge0}
\end{align}
is a fBm with Hurst parameter $\theta$, defined  by
   \begin{equation}\label{eq theta}
   \theta:=\frac{1}{2}-\frac{1-H}{\alpha}.
\end{equation}
Here, the constant $\kappa$ is given by
   \begin{equation}\label{eq kappa}
  \kappa:=\left(\frac{ H\Gamma(2H)\Gamma(1-2\theta)\sin(\pi H)}{ \pi\alpha\theta}\right)^{\frac12}.
  \end{equation}
 The  decomposition of  \eqref{eq decom}, together with   Chung's LIL of the fBm (e.g.,  \cite[Section]{LS2001}, \cite{Tal96}),
implies that for  every fixed $t\in (0,\infty)$ and $x\in\RR$,
	\begin{equation}\label{eq Chung t>0}
		\liminf_{\varepsilon\downarrow0} \sup_{h\in[0,\varepsilon]}\frac{|u(t+h, x)-u(t, x)|}{\e^{\theta}(\log\log \e^{-1})^{-\theta}}
		 = \kappa \lambda^{\theta},\ \ \ a.s.,
	\end{equation}
 where  $\lambda$  is the small ball constant of a fractional Brownian motion with index $\theta$ (e.g. \cite[Theorem 6.9]{LS2001}).

   However, the argument in the proof of Chung's LIL of $u(t, x)$  at $t>0$ does not   work in the case of $t=0$ and the problem on Chung's LIL for $u(t, x)$  at $t=0$ was left open in \cite{GSWX2025}.   This effort is complicated by the fact that, near $t = 0$, the random process $T$ is not smooth; in fact, $T$ and $F$ are equally smooth locally near $t = 0$.     When $\alpha=2$ and $H=\frac12$,  Khoshnevisan,   Kim,  and  Mueller \cite{KKM2024} recently  proved that  the small-ball probability of the rougher process $F$ to dominate that of the smoother Gaussian process $T$, and they adopted a localization idea of Lee and Xiao \cite{LX2023} to establish   Chung's  LIL of $u(t,x)$ at $t=0$.

   The objective of this   paper is to establish Chung's LIL  at    $t=0$ for the solution $u(t, x)$  of the linear stochastic fractional heat equation \eqref{SFHE},    for any fixed $x\in \mathbb R$, with parameters    $\alpha\in (1, 2]$ and $H\in (1-\alpha/2, 1)$.

   \begin{theorem}\label{pr:Chung:u}
	For any fixed $x\in\RR$,
	\begin{equation}\label{pr:Chung}
		\liminf_{\varepsilon\downarrow0} \sup_{t\in[0,\varepsilon]}\frac{|u(t,x)|}{\e^{\theta}(\log\log \e^{-1})^{-\theta}}
		 = \kappa\lambda^{\theta},\ \ \ a.s.,
	\end{equation}
where   $\lambda$  is the small ball constant of a fractional Brownian motion with index $\theta$.
\end{theorem}

 The lower bound in \eqref{pr:Chung} is derived from a small ball probability result (Lemma \ref{lem smb}), following the classical method of Monrad and Rootz\'en \cite{MR1995} and Talagrand \cite{Tal96}.  The upper bound, however, requires a more intricate analysis. To establish this upper bound, we adapt the localization approach introduced by Lee and Xiao \cite{LX2023} and Khoshnevisan et al. \cite{KKM2024}.

The paper is structured as follows. In Section 2, we  recall some properties of the solution,  and  Borell's  inequality  and  Dudley's metric entropy theorem.
  In Section 3, we prove Theorem \ref{pr:Chung:u} by analyzing the qualities of the approximation errors.

 \section{Preliminaries}\label{lemmas}
 \subsection{Properties of the solution}
 Let us recall  some notations from \cite{HHLNT2017}.
    Let  $ \cD=\cD(\RR) $ be  the space of real-valued infinitely differentiable  functions with compact support on $\mathbb{R}$.
    The Fourier transform of a function $f\in\cD$ is defined as
     $$
      \cF f(\xi):=\int_{\RR} e^{-i\xi x}f(x) dx.
    $$
 Let $\mathcal{D}(\mathbb{R}_{+}\times \mathbb{R})$  denote   the space of real-valued infinitely differentiable functions with compact support in $\mathbb{R}_{+}\times \mathbb{R}$ and let $W$
 be a zero-mean Gaussian family $\{W(\phi), \phi\in\cD(\RR_{+}\times \RR)\}$,  whose covariance structure is given by
     \begin{equation}\label{eq H product}
     \EE\big[ W(\phi)W(\psi)\big]=C_{H}\int_{\RR_{+}\times \RR} \cF \phi(s,\xi) \overline{\cF \psi(s,\xi)}\cdot |\xi|^{1-2H}d\xi ds,
     \end{equation}
     where $H\in(0,1)$, $C_{H}=\frac1{2\pi}\Gamma(2H+1)\sin(\pi H)$
     and $\cF \phi(s,\xi)$ is the Fourier transform of $\phi(s,x)$ with respect to the spatial variable $x$.
    Let $\mathcal H$  be the completion of $\mathcal D(\RR_+\times\RR)$ with the inner product
    \[
    \langle\phi,\psi\rangle_{\mathcal H}:=\EE\big[ W(\phi)W(\psi)\big].
    \]
     The mapping $\varphi\mapsto W(\varphi)$ can be extended to all $\varphi\in\mathcal H$.

       According to  Balan   \cite[Remark 5.4]{B12} and \cite[Proposition 2.2]{LQW2025}, we have the following results about   the solution to Equation  \eqref{SFHE}.
        \begin{proposition}\label{solution}   Assume $\alpha\in (1,2]$ and $\frac{2-\alpha}{2}<H<1$. Then, the equation \eqref{SFHE}  admits  a unique  solution $\{u(t,x); (t,x)\in \mathbb R_+\times \mathbb R\}$. Moreover, the following items hold.
\begin{itemize}
\item[(a)] For any $t\ge0$ and $x\in \mathbb R$,
\beq\label{u-L2}
\|u(t,x)\|_{L^2(\Omega)} =\sqrt{c_{2,1}}t^{\theta},
\nneq
where $c_{2,1}=\frac{C_{H}} {2H+\alpha-2}  2^{ \frac{2H+\alpha-2}{\alpha} } \Gamma\left(\frac{2-2H}{\alpha}\right)$.
\item[(b)]   For any fixed $x\in\mathbb R$, there exists $c_{2,2}>0$  such that
for any  $t, s\in \RR_+$,
\begin{equation}\label{equ-metric-d1}
\|u(t,x)-u(s,x)\|_{L^2(\Omega)}\le c_{2,2}|t-s|^{\theta}.
\end{equation}

\item[(c)] For any fixed $x\in \mathbb R$, the process $\{u(t, x)\}_{t\ge0}$ is a  Gaussian self-similar   process with index $\theta$ given by \eqref{eq theta}.
\end{itemize}
      \end{proposition}
\begin{proof}   The results in (a) and (b) are taken from    \cite[Proposition 2.2,  Lemma 2.1]{LQW2025}. Now, we give the proof for (c).
 For any fixed $x\in \mathbb R$, $\rho>0$ and  $0<s<t$,  by   \eqref{eq H product} and the change of variables, we have  \begin{equation}\label{cov:u}
\begin{split}
&\ee \left[\rho^{-\theta}u\left(\rho t,x\right)
           \rho^{-\theta}u\left(\rho s,x\right)\right]\\
=&\, C_{H} \rho^{-2\theta} \int_0^{\rho s}\int_{\RR}
     \mathcal F G\left(\rho t-r, x-\cdot\right)(\xi)
     \overline{\mathcal F G\left(\rho s-r, x-\cdot\right)(\xi)}
     |\xi|^{1-2H}d\xi dr\\
=&\, C_{H}\rho^{-2\theta}
     \int_0^{\rho s}\int_{\RR}
     e^{-(\rho t+\rho s-2r)|\xi|^\alpha}|\xi|^{1-2H}d\xi dr\\
=&\, C_{H}\int_0^{s}\int_{\RR} e^{-(t+s-2r)|\xi|^\alpha}|\xi|^{1-2H}d\xi dr\\
=&\, \ee\left[u(t, x)u(s, x)\right].
 \end{split}
 \end{equation}
 The proof is complete.
\end{proof}

 Let us recall   a small ball estimate result  for the process $\{u(t, x)\}_{t\ge0}$ for any fixed $x\in \mathbb R$.
 \begin{lemma} \cite[Theorem 4.3]{GSWX2025}\label{lem smb} For every  fixed $x\in\RR$,
\begin{equation}\label{pr:u}
\lim\limits_{\varepsilon\downarrow0}\varepsilon^{\frac1\theta}\log
\bp\left(\sup_{t\in[0,1]}|u(t,x)|\le\varepsilon\right)
	=-\kappa^{\frac1\theta}\lambda,
\end{equation}
where $\lambda$  is the small ball constant of a fractional Brownian motion with index $\theta$ .
     \end{lemma}

 \subsection{Borell's  inequality  and  Dudley's metric entropy theorem}

 For a given Gaussian process $\{X(t)\}_{t \in \mathbb{T}}$, let $d_X$ be the canonical metric defined by
\[
d_X(s, t) := \|X(t) - X(s)\|_{L^2(\Omega)} = \left( \mathbb{E}[|X(t) - X(s)|^2] \right)^{1/2}, \quad s, t \in \mathbb{T}.
\]
Let $N(\mathbb{T}, d_X; \varepsilon)$ denote the smallest number of open $d_X$-balls of radius $\varepsilon$ required to cover $\mathbb{T}$. The $d_X$-diameter of $\mathbb{T}$ is defined as $$D_X(\mathbb{T}) := \sup\left\{d_X(s, t) : s, t \in \mathbb{T}\right\}.$$

 \begin{lemma} {\rm (Borell's  inequality,  \cite[Theorem 2.1]{adler})}\label{Borell}
Let $\{X_t,t\in \mathbb T\}$  be a centered separable Gaussian process on some   index set
$\mathbb T$
with almost surely bounded sample paths. Then $\EE\left[\sup_{t\in  \mathbb T} X_t\right]<\infty$, and for all $\lambda>0$,
	\begin{equation}\label{lem3}
	\bp\left(\lt|\sup_{t\in \mathbb T} X_t-\EE\left[\sup_{t\in \mathbb T} X_t\right]\rt|>\lambda\right)\leq 2\exp\lc-\frac {\lambda^2}{2\sigma_{\mathbb T}^2}\rc,
	\end{equation}
	where $\sigma_{\mathbb T}^2:=\sup_{t\in \mathbb T}\EE\left[X_t^2\right]$.
 \end{lemma}

\begin{lemma}{\rm (Dudley's metric entropy theorem, \cite[Theorem 6.1]{L1996})}\label{Dudley}
There exists a positive constant $c_{2,3}$ such that for all Gaussian processes $\left\{X(t)\right\}_{t\in \mathbb T}$,
$$
\EE\left[\sup_{t\in \mathbb T} X(t)\right]\leq c_{2,3} \int_0^{D_X(\mathbb T)}
\sqrt{\log N\left( \mathbb T,d_X;\varepsilon\right)}d\varepsilon.
$$
\end{lemma}

\section{Proof  of Theorem \ref{pr:Chung:u}}

\subsection{Proof  of Theorem \ref{pr:Chung:u}}
\begin{proof}[Proof  of Theorem \ref{pr:Chung:u}]
\noindent{(1) {\bf Lower bound:}}   In this step, we prove the lower bound:
 \begin{equation}\label{u:lower}
\liminf_{\varepsilon\downarrow0}\sup_{t\in[0,\varepsilon]}
		 \frac{|u(t,x)|}{\e^{\theta}(\log\log \e^{-1})^{-\theta}}\geq
		 \kappa\lambda^{ \theta },\ \ \text{a.s.}
\end{equation}

For any $c>0$ and integer $n\ge1$, let $r_n:=e^{-cn}$.   For any fixed $x\in \mathbb R$ and $0<\delta<\lambda$, define
 $$
A_n:=\left\{ \sup_{t\in[0,r_n]}|u(t,x)|    \leq  \kappa \left( \frac{ \delta  r_n}
     {  \log \log (1/r_n) }\right)^{\theta} \right\}.
 $$
  By the self-similarity of $\{u(t, x)\}_{t\ge0}$ and    \eqref{pr:u},   there exists a constant 
  $N_0>0$ such that for any $n\ge N_0$,
 $$
 \mathbb P\left(A_n\right) \le (nc)^{ -  \frac{\lambda+\delta}{2\delta}}.
 $$
 Consequently,
 $$
 \sum_{n\in \mathbb N} \mathbb P\left(A_n\right) <\infty.
 $$
Using the Borel-Cantelli lemma and the  arbitrariness of $\delta\in (0,\lambda)$, we have
 \begin{equation*}
  \liminf_{n\rightarrow\infty}\sup_{t\in[0,r_n]}\frac{|u(t,x)|}  { {r_n^{\theta}(\log\log r_n^{-1})^{-\theta}} }\geq
\kappa \lambda^{\theta},\ \ \text{a.s.}
\end{equation*}
This, together with a   monotonicity argument, implies that
 \begin{align*}
 \begin{split}
  \liminf_{r\rightarrow0}\sup_{t\in[0,r]}\frac{|u(t,x)|}  {{r^{\theta}(\log\log r^{-1})^{-\theta}} } \geq   \,    \liminf_{n\rightarrow\infty}\sup_{t\in[0,r_{n+1}]}\frac{|u(t,x)|}  {{r_n^{\theta}(\log\log r_n^{-1})^{-\theta}}} \ge \kappa e^{-c}    \lambda^{\theta}, \   \ \text{a.s.}
  \end{split}
  \end{align*}
Using the arbitrariness of $c>0$, we obtain the lower bound \eqref{u:lower}.

\noindent{(2) {\bf Upper bound}}.  In this step, we prove the upper bound:
  \begin{equation}\label{u:upper}
\liminf_{r\rightarrow0}\sup_{t\in[0,r]}\frac{|u(t,x)|}  {{r^{\theta}(\log\log r^{-1})^{-\theta}} } \le
		 \kappa\lambda^{ \theta },\ \ \text{a.s.}
\end{equation}
Here, we adopt a localization idea of  Lee and Xiao \cite{LX2023} and  Khoshnevisan et al. \cite{KKM2024}.

Define a family   of Gaussian random fields  $\{u_n\}_{n\in \mathbb N}$ by setting
 \begin{equation}\label{u_n}
    u_n(t,x):= \int_{[t_{n+1},t)\times\mathbb{R}} G(t-s,x-y) \, W(ds,dy), \quad (t,x) \in [t_{n+1},t_n] \times \RR,
    \end{equation}
     where
     \beq\label{t_n}
t_n:=\exp\left(-n^{1+\beta}\right), \ \ \ n\in \mathbb N,
\nneq
with $\beta$ being a constant  whose value will be chosen later.

For  simplicity,  define  $\psi: (0, 1/e) \to (0, \infty)$ by
\begin{equation}\label{def:psi}
\psi(t) := \left( \frac{t}{\log\log(1/t)} \right)^{\theta}.
\end{equation}

Notice that
\begin{equation}\label{liminf:u:upper:ineq}
\begin{split}
    \liminf_{n\rightarrow\infty}\sup_{t\in[0,t_n]}\frac{|u(t,x)|}{\psi(t_n)}
   \leq&\, \limsup_{n\rightarrow\infty}\sup_{t\in[0,t_{n+1}]}\frac{|u(t,x)|}{\psi(t_n)}\\
   &\, \, +
 \limsup_{n\rightarrow\infty}\sup_{t\in[t_{n+1},t_n]}\frac{|u(t,x)-u_n(t,x)|}{\psi(t_n)}  \\
&\,\, +\liminf_{n\rightarrow\infty}\sup_{t\in[t_{n+1},t_n]}\frac{|u_n(t,x)|}{\psi(t_n)}.
\end{split}
\end{equation}
 By Lemmas  \ref{lem:P(u:n+1)} and \ref{lem:Var(u-u_n)} below,  and the Borel-Cantelli lemma, we  have
\begin{equation}\label{u:0:tn1}
\begin{split}
 \limsup_{n\rightarrow\infty}\sup_{t\in[0,t_{n+1}]}\frac{ | u(t,x)| }  {\psi(t_n)}
= \limsup_{n\rightarrow\infty}\sup_{t\in[t_{n+1},t_n]}\frac{ | u(t,x) - u_n(t,x) |  } {\psi(t_n)}=0,\ \ \text{ a.s.}
\end{split}
\end{equation}

We now turn to the third term on the right-hand side of inequality \eqref{liminf:u:upper:ineq}.    Applying   Lemma \ref{lem:small-ball-u_n} below with  $\gamma:= \kappa(1+2\beta)^\theta \lambda^\theta$,
$$
\lim_{n\rightarrow\infty}\frac{1}{\log n}
		\log\bp\left( \sup_{t\in[t_{n+1},t_n]}| u_n(t,x) | \le  \kappa(1+2\beta)^\theta \lambda^\theta \psi(t_n)\right)
=- 1+\frac{\beta}{1+2\beta}.
$$
Consequently,
 \begin{equation}\label{eq probab un}
\sum_{n\in\mathbb N}\mathbb P \left( \sup_{t\in[t_{n+1},t_n]}| u_n(t,x) | \le  \kappa(1+2\beta)^\theta \lambda^\theta \psi(t_n)\right) =\infty.
\end{equation}
The independence of the families $\{u_n(t, x): t_{n+1} \le t \le t_n\}_{n \ge 1}$, together with  \eqref{eq probab un} and the Borel-Cantelli lemma, yields that
 \begin{equation}\label{un:tn1:tn}
\liminf_{n\to\infty}\sup_{t\in[t_{n+1},t_n]}\frac{|u_n(t,x)| }{\psi(t_n)}\leq
	  \kappa (1+2\beta)^{ \theta }  \lambda ^{ \theta },                                            \ \ \text{a.s.}
\end{equation}
Combining inequalities \eqref{liminf:u:upper:ineq}, \eqref{u:0:tn1}, and \eqref{un:tn1:tn}, and passing to the limit as $\beta \downarrow 0$, we obtain \eqref{u:upper}.    The proof is complete.
\end{proof}

   \subsection{Some important lemmas} In this part, we establish some lemmas used in the proof of Theorem \ref{pr:Chung:u}.
  Recall the sequence $t_n$ defined by \eqref{t_n}.  Using the mean-value theorem, we have that for all $n\in\mathbb N$,
\beq\label{t/t}
\frac{t_{n+1}}{t_n}\leq\exp\left(-(1+\beta)n^\beta\right).
\nneq
 \begin{lemma}\label{lem:P(u:n+1)}
For any fixed $x\in\RR$ and for every $\delta>0$, there exist $c_{3,1}>0$ and $N_1(\delta)>0$, depending on $\delta$, such that for all $n\geq N_1(\delta)$,
$$
\bp\left( \sup_{t\in[0,t_{n+1}]}|u(t,x)| \geq\delta\psi(t_n) \right)
\leq c_{3,1}\exp \left( -\frac{ \exp\left( 2\theta(1+\beta) n^{\beta} \right) }
{c_{3,1} \left( \log n\right)^{2\theta}} \right).
$$
\end{lemma}
 \begin{proof}
For every $\varepsilon>0$ and $i\in \mathbb N$, let $s_i= i\left(\frac{\varepsilon}{c_{2,2}}\right )^{\frac1\theta}$.
By \eqref{equ-metric-d1},
  $$d_{u}(s_{i-1},s_i) \leq \varepsilon.$$
     Consequently,    the covering number $N([0,t_{n+1}],d_{u}; \varepsilon)$     satisfies  that
\begin{equation}\label{upper:N:T}
N([0,t_{n+1}],d_{u}; \varepsilon)\leq   2 t_{n+1} \left(\frac{\varepsilon}{c_{2,2}}\right )^{-\frac1\theta}.
\end{equation}
Using  \eqref{equ-metric-d1} again, the diameter of $[0,t_{n+1}]$ under the canonical metric of $\{u(t, x)\}_{t\ge0}$ is
\begin{equation}\label{Du}
\begin{split}
D_{u}([0,t_{n+1}]) \leq c_{2,2}t_{n+1}^{\theta}.
\end{split}
\end{equation}
By \eqref{upper:N:T},  \eqref{Du},  Lemma \ref{Dudley}, and the change of variables, there exists  $c_{3,2}>0$  such that
\begin{equation}\label{Dudley-u}
\begin{split}
\EE\left[ \sup_{t\in[0,t_{n+1}]}|u(t,x)| \right]
\leq& \, c_{2,3}\int_0^{D_{u}([0,t_{n+1}])} \sqrt{\log N\left( [0,t_{n+1}],d_{u};\varepsilon\right)}d\varepsilon\\
\leq& \, c_{2,3}\int\limits_0^{c_{2,2}t_{n+1}^{\theta}}
        \sqrt{\log\left(2t_{n+1}(c_{2,2})^{\frac1\theta}\varepsilon^{-\frac1\theta}
        \right)}d\varepsilon\\
=&\,c_{2,3} c_{2,2}\theta2^\theta    t_{n+1}^\theta
          \int_0^{\frac12} \omega^{\theta-1}\sqrt{-\log \omega}d\omega\\
\leq&\, c_{3,2}t_{n+1}^{\theta}.
\end{split}
\end{equation}
Here, the following inequality is used in the last step:  For any $q<1$ and $0<\eta<1$, there exists an explicit constant $c$ depending  on $q,\eta$ such that for any $0<A<\eta<1$,
\begin{equation}\label{ineq-1}
\int_0^{A}\omega^{-q}\sqrt{-\log \omega}d\omega\leq c  A^{1-q}\sqrt{-\log A},
\end{equation}
 see \cite[Lemma B.3]{CT2025})

   For every $\delta>0$, by \eqref{def:psi},  \eqref{t/t}, and \eqref{Dudley-u},  there exists $N_1(\delta)>0$, depending on $\delta$, such that for all $n\geq N_1(\delta)$,
\begin{equation}\label{E:tx}
\EE\left[ \sup_{t\in[0,t_{n+1}]}|u(t,x)| \right]\leq\frac{\delta}{2} \psi(t_n).
\end{equation}
By Lemma \ref{Borell}, \eqref{u-L2},  \eqref{E:tx},   and \eqref{t/t},
there exists $c_{3,3}>0$  such  that for all  $n\geq N_1(\delta)$,
\begin{equation}\label{pp:u}
\begin{split}
\bp\left( \sup_{t\in[0,t_{n+1}]}|u(t,x)| \geq\delta\psi(t_n) \right)
\leq&\, \bp\left( \sup_{t\in[0,t_{n+1}]}|u(t,x)|-\EE\left[ \sup_{t\in[0,t_{n+1}]}|u(t,x)|\right]
       \geq\frac{\delta\psi(t_n)}{2}\right)\\
\le &\, 2\exp\left(-\frac{\delta^2}{8c_{2,1}(\log\log (1/t_n))^{2\theta}} \left(\frac{t_n}{t_{n+1}}\right)^{2\theta} \right)\\
\leq&\, 2\exp\left( -\frac{ \exp\left( 2\theta(1+\beta) n^{\beta} \right)}
{ \left( 8c_{2,1}(1+\beta)^{2\theta}/\delta^2 \right)(\log n)^{2\theta}}  \right)\\
\leq&\, c_{3,3}\exp \left( -\frac{ \exp\left( 2\theta(1+\beta) n^{\beta} \right) }
{c_{3,3} \left( \log n \right)^{2\theta}} \right).
\end{split}
\end{equation}
  The proof is complete.
\end{proof}

Recall the definition of $u_n$ in \eqref{u_n}.  Next, we establish a sharp bound on the approximation error $|u(t,x)- u_n(t,x)|$.
 \begin{lemma}\label{lem:Var(u-u_n)}
For any fixed $x\in\RR$ and for every $\delta>0$, there exist $c_{3,4}>0$ and $N_2(\delta)>0$, depending on $\delta$, such that for all $n\geq  N_2(\delta)$,
\begin{align}\label{eq u-u_n}
 \bp\left(\sup_{t\in[t_{n+1},t_n]}|u(t,x)- u_n(t,x)| \ge \delta\psi(t_n)\right)
   \le c_{3,4}\exp \left( -\frac{\exp\left( 2\theta(1+\beta) n^{\beta} \right)}
 {c_{3,4} \left( \log n \right)^{2\theta}} \right).
\end{align}
 \end{lemma}
\begin{proof}
The proof of \eqref{eq u-u_n} is similar to that of   Lemma \ref{lem:P(u:n+1)},   using  Dudley's metric entropy theorem and Borell's inequality.   For brevity, denote
  \begin{equation}\label{def:Y_n}
Y_n(t, x):=u(t,x)-u_n(t, x).
\end{equation}
Let us   give the upper bound for the  covering number $N\left([t_{n+1},t_n],d_{Y_n};\varepsilon\right)$ and the diameter $D_{Y_n}([t_{n+1},t_n])$ under the canonical metric of $\{Y_n(t, x)\}_{t\ge0}$, respectively.

For any  $t_n\geq t>s\geq t_{n+1}$, since $Y_n(t, x)-Y_n(s, x)$ is independent of $u_n(t, x)-u_n(s, x)$, by \eqref{equ-metric-d1} and \eqref{def:Y_n}, we have
 \begin{equation}\label{Dudley-(u-u_n)}
\| Y_n(t,x)- Y_n(s,x)\|_{L^2(\Omega)}\leq \, c_{2,2}|t-s|^{\theta}.
\end{equation}
 Similarly to the proof \eqref{upper:N:T},    the covering number of  $N\left([t_{n+1},t_n], d_{Y_n}; \varepsilon\right)$  satisfies
that
  \begin{equation}\label{N:dY}
N\left([t_{n+1},t_n],d_{Y_n}; \varepsilon\right) \leq c_{3,5}(t_n-t_{n+1})\varepsilon^{-\frac1\theta}.
\end{equation}

Using   \eqref{eq H product} and   the change of variables
$\eta= 2(t-s) \xi^{\alpha}$, we have
\begin{equation}\label{eq Y_n moment}
\begin{split}
\|Y_n(t,x)\|_{L^2(\Omega)}^2
=&\,2C_{H}\int_0^{t_{n+1}}\int_0^{\infty}e^{-2(t-s)\xi^{\alpha}}\xi^{1-2H}d\xi ds \\
=&\, \frac{C_{H}}{\alpha}2^{\frac{2H+\alpha-2}{\alpha}}\int_0^\infty e^{-\eta }\eta^{\frac{2-2H}{\alpha}-1} d\eta \int_0^{t_{n+1}} (t-s)^{\frac{2H-2}{\alpha}} ds \\
=&\,  c_{2,1}\left(t^{2\theta}-(t-t_{n+1})^{2\theta}\right)\\
\le &\, c_{2,1}  t_{n+1}^{2\theta}.
\end{split}
\end{equation}
where the elementary $a^{2\theta}-b^{2\theta}\le   (a-b)^{2\theta}$ is used for any $a>b>0$ and $\theta\in (0,1/2)$.

By  \eqref{t/t} and \eqref{eq Y_n moment}, we have that for all $n\in\mathbb N$,
\begin{equation}\label{upper:dY}
D_{Y_n}([t_{n+1},t_n])
\leq  \,  2 c_{2,1}^{\frac12}t_n^{\theta}\exp\left( -\theta(1+\beta) n^{\beta} \right).
\end{equation}

Using   \eqref{N:dY}, \eqref{upper:dY},  and Lemma \ref{Dudley}, we have
\begin{equation}\label{Dudleys:u-u_n}
\begin{split}
\EE \left[ \sup_{t\in[t_{n+1},t_n]}|Y_n(t,x)|\right]
\leq&\,  c_{2,3}\int_0^{D_{Y_n}([t_{n+1},t_n])}
      \sqrt{\log N\left( [t_{n+1},t_n],d_{Y_n};\varepsilon\right)}d\varepsilon\\
\leq&\, c_{2,3}\int\limits_0^{     2c_{2,1}^{\frac12}t_n^{\theta}
       \exp\left( -\theta(1+\beta) n^{\beta} \right)      }
     \sqrt{\log\left(c_{3,5} (t_n-t_{n+1}) \varepsilon^{-\frac1\theta}\right)}d\varepsilon\\
\le &\, c_{2,3} \theta   \left(c_{3,5}(t_n-t_{n+1})\right)^{\theta}
     \int_0^{A_n}\omega^{\theta-1}\sqrt{-\log \omega}d\omega,
\end{split}
\end{equation}
where  the change of variables $\omega:=c_{3,5}^{-1}(t_n-t_{n+1})^{-1}\varepsilon^{\frac1\theta}$ is used, and
$$A_n:=c_{3,5}^{-1}(4c_{2,1})^{\frac1{2\theta}}(t_n-t_{n+1})^{-1}t_n\exp\left( -(1+\beta)n^{\beta} \right).$$

   By   \eqref{t/t} and  \eqref{ineq-1}, there exist $c_{3,6}>0$ and $N_3>0$ such that for all $n\geq N_3$,
\begin{equation}\label{Dudleys:u-u_n:1}
\begin{split}
\EE \left[ \sup_{t\in[t_{n+1},t_n]}|Y_n(t,x)|\right]
\leq& \,  c_{3,6}t_n^{\theta}
    \exp\left( -\theta(1+\beta) n^{\beta} \right)n^{(1+\beta)/2}.
\end{split}
\end{equation}
Similarly as the proof of \eqref{pp:u}, using  Borell's inequality and \eqref{Dudleys:u-u_n:1}, there exist $c_{3,7}>0$ and $N_4(\delta)$, depending on $\delta$, such that for all
$n\geq N_4(\delta)$,
\begin{equation*}
\mathbb P\left( \sup_{t\in[t_{n+1},t_n]}|Y_n(t,x)| \ge \delta\psi(t_n)\right)
\le c_{3,7}\exp \left( -\frac{ \exp\left( 2\theta(1+\beta)n^{\beta} \right)}
   {c_{3,7} \left( \log n \right)^{2\theta}} \right).
\end{equation*}
 The proof is complete.
\end{proof}

Next, we derive a  small-ball estimate for $u$ on $[t_{n+1},  t_n]$.
\begin{lemma}\label{lem:small-ball-u}
For any fixed $x\in\RR$ and for every $\gamma>0$,
	\begin{align}	\label{eq sb u tn}
		\lim_{n\to\infty} \frac{1}{\log n}
		\log\bp\left(\sup_{t\in[t_{n+1},t_n]} | u(t,x)|\le \gamma\psi(t_n)\right)
		=-\lambda \left(  \frac{ \kappa}{\gamma} \right)^{\frac1\theta} (1+\beta).
	\end{align}
\end{lemma}

\begin{proof}
\noindent{\bf (i) Lower bound.}
 By  \eqref{pr:u} and  \eqref{def:psi}, we have that for every $\gamma>0$,
\begin{equation}\label{liminf:u}
\begin{split}
&\, \liminf_{n\rightarrow\infty}\frac 1 {\log n}
    \log\bp\left( \sup_{t\in[t_{n+1},t_n]}|u(t,x)|\leq\gamma\psi(t_n) \right)\\
     \ge     &\, \liminf_{n\rightarrow\infty}\frac 1 {\log n}
    \log\bp\left( \sup_{t\in[0,t_n]}|u(t,x)|\leq\gamma\psi(t_n) \right)\\
=&\,  \liminf_{n\rightarrow\infty}\frac 1 {\log n}
   \log\bp\left\{ \sup_{t\in[0,1]}|u(t,x)|\leq\gamma (\log\log(1/t_n))^{-\theta} \right\}\\
  = &\,  -\lambda \left(  \frac{ \kappa}{\gamma} \right)^{\frac1\theta} (1+\beta).
\end{split}
\end{equation}

\noindent{\bf (ii) Upper bound.}
For any $\delta\in(0,\gamma)$, by Lemma \ref{lem:P(u:n+1)}, we have that
\begin{equation}\label{limsup:u}
\begin{split}
& \limsup_{n\rightarrow\infty}\frac 1 {\log n} \log  \bp\left( \sup_{t\in[t_{n+1},t_n]}|u(t,x)|\leq\gamma\psi(t_n) \right)\\
\le &\,  \max\Bigg\{\limsup_{n\rightarrow\infty}\frac 1 {\log n} \log     \bp\left( \sup_{t\in[0,t_n]}|u(t,x)|\leq(\gamma+\delta)\psi(t_n) \right), \\
   & \ \ \ \ \   \ \ \  \ \   \limsup_{n\rightarrow\infty}\frac 1 {\log n} \log  \  \bp\left( \sup_{t\in[0,t_{n+1}]}|u(t,x)|\geq\delta\psi(t_n) \right)\Bigg\}\\
=&\, \limsup_{n\rightarrow\infty}\frac 1 {\log n}
\log\bp\left( \sup_{t\in[0,1]}|u(t,x)|\leq(\gamma+\delta) (\log\log(1/t_n))^{-\theta} \right)\\
=&\, - \lambda\left(  \frac{\kappa}{\gamma+\delta} \right)^{\frac1\theta} (1+\beta)  .
\end{split}
\end{equation}
Putting \eqref{liminf:u} and \eqref{limsup:u} together, and  letting $\delta\downarrow0$, we get   \eqref{eq sb u tn}. The proof is complete.
\end{proof}

Next, we derive a  small-ball estimate for $u_n$ on $[t_{n+1},  t_n]$, which plays an important role in the proof of Theorem \ref{pr:Chung:u}.

\begin{lemma}\label{lem:small-ball-u_n}
  For any fixed $x\in\mathbb R$ and for every $\gamma>0$,
	\begin{equation}\label{eq sb un}
		\lim_{n\to\infty} \frac{1}{\log n}
		\log\bp\left( \sup_{t\in[t_{n+1},t_n]}| u_n(t,x) | \le \gamma\psi(t_n)\right)
		=-\lambda \left(  \frac{ \kappa}{\gamma} \right)^{\frac1\theta} (1+\beta).
	\end{equation}
  \end{lemma}

 \begin{proof}
 \noindent{\bf (i) Upper bound.}
For any $\delta\in(0,\gamma)$, by Lemmas \ref{lem:Var(u-u_n)} and  \ref{lem:small-ball-u}, we have that
\begin{equation}\label{limsup:u_n}
\begin{split}
& \limsup_{n\rightarrow\infty}\frac 1 {\log n} \log  \bp\left( \sup_{t\in[t_{n+1},t_n]}|u_n(t,x)|\leq\gamma\psi(t_n) \right)\\
\le &\,  \max\Bigg\{\limsup_{n\rightarrow\infty}\frac 1 {\log n} \log     \bp\left( \sup_{t\in[t_{n+1},t_n]}|u(t,x)|\leq(\gamma+\delta)\psi(t_n) \right), \\
   & \ \ \ \ \   \ \ \  \ \   \limsup_{n\rightarrow\infty}\frac 1 {\log n} \log  \  \bp\left( \sup_{t\in[t_{n+1},t_n]}|u(t,x)-u_n(t,x)|\geq\delta\psi(t_n) \right)\Bigg\}\\
=&\,   -\lambda \left(  \frac{ \kappa}{\gamma+\delta} \right)^{\frac1\theta} (1+\beta).
\end{split}
\end{equation}

\noindent{\bf (ii) Lower bound.}  For any $\delta\in(0,\gamma)$, by Lemma \ref{lem:Var(u-u_n)}, we have that for all  $n\geq N_2(\delta)$,
\begin{equation*}
\begin{split}
&\bp\left( \sup_{t\in[t_{n+1},t_n]}|u(t,x)|\leq(\gamma-\delta)\psi(t_n) \right)\\
\leq&\, \bp\left( \sup_{t\in[t_{n+1},t_n]}|u_n(t,x)|\leq\gamma\psi(t_n) \right)
    + \bp\left( \sup_{t\in[t_{n+1},t_n]}|u(t,x)-u_n(t,x)|\geq\delta\psi(t_n) \right)\\
\leq&\, \bp\left( \sup_{t\in[t_{n+1},t_n]}|u_n(t,x)|\leq\gamma\psi(t_n) \right)
    +c_{3,4}\exp \left( -\frac{\exp\left( 2\theta(1+\beta) n^{\beta} \right)}
      {c_{3,4} \left( \log n \right)^{2\theta}} \right).
\end{split}
\end{equation*}
This, together with   Lemma \ref{lem:small-ball-u}, implies that
\begin{equation}\label{liminf:u_n}
\begin{split}
 & \liminf_{n\rightarrow\infty}\frac 1 {\log n}
       \log\bp\left( \sup_{t\in[t_{n+1},t_n]}|u_n(t,x)|\leq\gamma\psi(t_n) \right)\\
        \ge &\, \liminf_{n\rightarrow\infty}\frac 1 {\log n}
      \log\bp\left( \sup_{t\in[t_{n+1},t_n]}|u(t,x)|\leq(\gamma-\delta)\psi(t_n) \right)\\
=&\,   - \lambda \left(  \frac{ \kappa}{\gamma-\delta} \right)^{\frac1\theta}
    (1+\beta).
\end{split}
\end{equation}
Putting \eqref{limsup:u_n} and \eqref{liminf:u_n} together, and letting $\delta\downarrow 0$, we obtain \eqref{eq sb un}. The proof is complete.
\end{proof}

 \vskip0.3cm

\noindent{\bf Acknowledgments}   

The research of R. Wang is partially supported by the NSF of Hubei Province (No. 2024AFB683).

 \vskip0.3cm

\vskip0.3cm
\noindent{\bf Data availability}

No data was used for the research described in the article.

\end{document}